\documentclass{amsart}
\usepackage{amscd,enumerate}
\usepackage{amssymb,amsmath,amsthm}
\usepackage{latexsym,amsfonts,amscd}

\usepackage[utf8]{inputenc}
\usepackage[OT4]{fontenc}

\newcommand*{\R}{\ensuremath{\mathbb{R}}}

\newcommand*{\RR}{\ensuremath{\mathbb{R}}}

\newcommand*{\Sph}{\ensuremath{\mathbb{S}}}

\renewcommand*{\epsilon}{\varepsilon}

\newtheorem{theorem}{Theorem}
\newtheorem{definition}[theorem]{Definition}
\newtheorem{corollary}[theorem]{Corollary}
\newtheorem{lemma}[theorem]{Lemma}
\newtheorem{fact}[theorem]{Fact}

\newtheorem{remark}[theorem]{Remark}
\newtheorem{example}[theorem]{Example}

\begin{document}
\title[The closed graph function from plane ...]{The closed graph function from plane with the network of discontinuity points}
\author{Michał Stanisław Wójcik}

\begin{abstract}
The main result of this paper states, that if a function $f:\R^2\to [0, +\infty)$ has a closed graph and the set of discontinuity points is a network (as defined by Kuratowski in Topology II, 61.IV), then the graph of $f$ is disconnected. It is also proven that this result can be easily generalised to a function $f:\R^2\to Y$ where $Y$ is a $\sigma$-locally compact metrisable space.
\end{abstract}

\maketitle

\subsection*{Motivation}
It is known that, for a function $f:\R\to \R$ with a closed graph, a sufficient and necessary condition for being continuous is the connectedness of the graph \cite{Bur}. In 2001, Michał R. Wójcik and I stated the question whether this result can be extended to $f:\R^2\to\R$ \cite[9]{wojcik1} – this problem was then propagated by Cz. Ryll-Nardzewski. The answer to this question is negative. The first known discontinuous function $f:\R^2\to\R$ with a connected and closed graph was shown by J. Jel\'{\i}nek in \cite{JL}. The set of discontinuity points in the Jel\'{\i}nek construction is some kind of fractal with infinitely many connected components of its complement.

\subsection*{Result}
The main result of this paper states that if a function $f:\R^2\to [0, +\infty)$ has a closed graph and the set of discontinuity points is a network (Definition \ref{network}), then the graph of $f$ is disconnected. This result is formulated as Theorem \ref{discon}. It follows from Lemma \ref{too_many}, which is a clue point of the paper. The generalisation of Theorem \ref{discon} is possible due to Theorem \ref{di_gen} and is formulated as Corollary \ref{big_hit}.

\section{Notation and terminology}

\begin{definition}
Let $X,Y$ be topological spaces and $f:X\to Y$ be an arbitrary function.
\begin{enumerate} 
\item We will denote by $C(f)$ or $C_f$ the set of all points of continuity,
\item by $D(f)$ or $D_f$ -- set of all points of discontinuity. 
\item For $A\subset X$, by $f|A$ we will denote a restriction of $f$ to the subdomain $A$.
\end{enumerate} 
\end{definition}

In the context of a function $f:X\to Y$, we will not use a separate symbol to denote the graph of $f$, for $f$ itself, in terms of Set Theory, is a graph. So when we use Set Theory operations and relations with respect to $f$, they should be understood as operations and relations with respect to the graph. Whenever this naming convention might be confusing, we will add the word “graph”, e.g. “$f$ has a closed graph”.

\begin{definition}
Let $(X,d)$ be a metric space.\\
$B(x_0,r) = \{x\in X:d(x_0,x)<r\}$,\\
$\hat{B}(x_0,r) = \{x\in X:d(x_0,x)\leq r\}$,\\
$S(x_0,r) = \{x\in X:d(x_0,x)=r\}$.
\end{definition} 

\section{Functions with a closed graph}

At the beginning I will cite some results concerning closed graph functions:

\begin{theorem}\label{fcon}
If $X$ is a topological space, $Y$ is a compact space, $f:X\to Y$ and the graph of $f$ is closed, then $f$ is continuous.
\end{theorem}

(for proof: e.g. \cite[T2]{wojcik1})

\begin{theorem}\label{dense_con}
If $X$ is a Bair and Hausdorff space, $Y$ is a $\sigma$-locally compact space, $f:X\to Y$ and the graph of $f$ is closed, then C(f) is an open and dense subset of $X$.
\end{theorem}

(for proof: e.g. \cite[T2]{Dob})

\begin{lemma}\label{continuous_factory}
If $X$ is topological space, $Y$ is a locally compact space, ${f \colon X \to Y}$ has a closed graph, 
$A \subset X$, $x_0 \in A$, $f|A$
is continuous at $x$, and for each open neighbourhood $U$ of $x_0$,
there exists an open neighbourhood $G$ of $x_0$, such that for any $y \in G$
there exists $E \subset U$ such that $y \in E$, the image $f(E)$ is connected
and $E \cap A \not= \emptyset$, then $f$ is continuous at $x_0$.
\end{lemma}

\begin{proof}
\cite[T2]{wojcik1}
\end{proof}

\section{locally connected accessability}

\begin{definition}\label{algebra}
Let $X$ be a topological space, $G(X)$ be a family of all open sets in $X$ and $f:X\to [0, +\infty)$, $A\subset X$.\\
$A^f_\infty = \{x\in Clo(A):\forall_{M>0}\forall_{U\in G(X):x\in U}\exists_{y\in U\cap A} f(y) > M\}$\\
$A^f_0 = Clo(A)\setminus A^f_\infty$.
\end{definition}

There are two quick conclusions from the above definition.

\begin{corollary}
If $f:X\to [0, +\infty)$ and $A\subset X$, then $A^f_\infty$ is closed and $A^f_\infty\subset D_f$ and $A^f_0$ is open in a relative topology of $Clo(A)$.
\end{corollary}

\begin{corollary}\label{easy_con}
If $f:X\to [0, +\infty)$ has a closed graph, $A\subset X$, then $f|Clo(A)$ is continuous on $A^f_0$.
\end{corollary}
\begin{proof}
By Definition, for any $x\in A^f_0$ there exists a real number $M>0$ and an open neighbourhood $U$ of $x$, such that $f|U\cap Clo(A) :U\cap Clo(A)\to[0,M]$. By Theorem \ref{fcon},  $f|U\cap Clo(A)$ is continuous. Since $U$ is open, $f|Clo(A)$ is continuous at $x$.
\end{proof}

The above terminology looks a bit helpless. But if we join it with some information about how a given point is located relativelly to the set of continuity points, we will get a convinient tool for finding interesting properietes of closed graph functions.

\begin{definition}\label{access}
Let $X$ be a topological space and $A\subset X$. We will say that $x$ is locally connectedly accessible from $A\subset X$, iff for each open neighbourhood $U$ of $x$ there is an open neighbourhood $G$ of $x$, such that $G\subset U$ and for each $u,v\in G\cap A$ there is a connected set $E$, such that $u,v\in E\subset U\cap A$.\\
We will say that the set $E\subset X$ is locally connectedly accessible from A, iff each point of $E$ is locally connectedly accessible from $A$.  
\end{definition}

Obviously if $x$ is locally connectedly accessible from $A$, then $x\in Clo(A)$.

\begin{example}\label{trivial_example}
Let $B\subset R^n$ be the unit ball. Let $x\in \partial B$ and $V$ be an arbitrary open set that contains $x$. Then for $A=V\cap Int B$, $x$ is locally connectedly accessible from $A$. 
\end{example}
\begin{proof}
Take any open neighbourhood $U$ of $x$. Choose such $\delta > 0$, that $B(x,\delta)\subset U\cap V$. Put $G=B(x,\delta)$. Take any $u,v\in A\cap G$. Since $G$ and $IntB$ are convex, $G\cap B$ is convex as well, so the interval $[u,v]\subset G\cap IntB$. But $G\cap IntB = G\cap A$. Put $E=[u,v]$. Obviously $E\subset U\cap A$. Thus by Definition \ref{access}, $x$ is locally connectedly accessible from $A$.
\end{proof}

For $x\in A^f_\infty$, generally, we know only that there exists at least one net $A\ni x_n\to x$ such that $f(x_n)\to+\infty$. But if we assume additionally that the graph of $f$ is closed and $x$ is locally connectedly accesible from $A$, then we can expect that for each net $A\ni x_n\to x$, $f(x_n)\to+\infty$.

\begin{theorem}\label{exploding}
If $f:X\to [0, +\infty)$ has a closed graph, $A\subset C_f$, $x$ is locally contectedly accessible from $A$ and $x\in A^f_\infty$, then
$$\lim\limits_{A\ni u\to x}f(u) = +\infty.$$ 
\end{theorem}
\begin{proof}
Assume to the contrary that for each open neighbourhood $U$ of $x$, there exists $x_U\in U\cap A$, such that $f(x_U)\to f(x)$ in terms of net convergence. Let $B = \{z_U:U\text{ is an open neighbourhood of } x\}$.
We will prepare to applay Lemma \ref{continuous_factory}. Our space will be in this case $A\cup\{x\}$. Keep in mind that $f|B\cup \{x\}$ is continuous at $x$. Take any open neigbourhood $V$ of $x$. By locally conectedly accessibility, there is an open neighbourhood $G$ of $x$, such that for each $u,v\in G\cap A$ there is a connected set $E$, such that $u,v\in E\subset V\cap A$. Take any $u\in G\cap (A\cup\{x\})$. If $u\not = x$, we have such connected $E$, that $u,x_G\in E\subset A\cap V$. If $u = x$, put $E=\{x\}$. Since $A\subset C_f$, $f(E)$ is connected. Thus by Lemma \ref{continuous_factory}, the function $f|(A\cup\{x\})$ is continuous at $x$. This contradicts to $x\in A^f_\infty$.
\end{proof}

\begin{corollary}\label{col_exp}
If $f:X\to[0,+\infty)$ has a closed graph, $A \subset C_f$ and $x$ is locally contectedly accessible from $A$, and $x\in A^f_\infty \cap Clo(A^f_0)$, then
$$
\lim_{A^f_0\ni u \to x} f(u) = +\infty.
$$
\end{corollary}

\begin{theorem}\label{neither_or}
If $f:X\to [0,+\infty)$ has a closed graph, $E$ is a connected subset of $X$, $f|E$ is continuous and $E$ is locally connectedly accessible from $A\subset X$, then either $E\subset A^f_0$ or $E\subset A^f_\infty$.
\end{theorem}
\begin{proof}
Assume to the contrary that $E\cap A^f_0 \not=\emptyset$ and $E\cap A^f_\infty \not=\emptyset$. Since $E\subset Clo(A)$, $A_0 \cap E$ is open in the relative topology of $E$. By Definition \ref{algebra}, $E\subset A^f_\infty \cup A^f_0$.
We will show that $A_\infty \cap E$ is open in the relative topology of $E$.
Suppose that $A_\infty \cap E$ isn't open in the relative topology of $E$. Take some border point $x$ of $A_\infty \cap E$ (in relative topology). 
Since $f|E$ is continuous, we have such an open neighbourhood $V$ of $x$, that $|f(x) - f(z)| < 1$ for any $z\in V\cap E$. Take any open neighbourhood $U$ of $x$. Since $x$ is a border point of $A_\infty \cap E$ (in relative topology), we have $v\in A^f_0 \cap E\cap (U\cap V)$. By Definition \ref{algebra}, there exists $u\in A\cap (U\cap V)$ such that $|f(u)-f(v)|<1$. Since $|f(u) - f(v)| < 1$, $|f(u) - f(x)| < 2$. Since $U$ was arbitrary chosen, it's not true that $\lim_{A\ni u \to x} f(u) = +\infty$, which contradicts Theorem \ref{exploding}. We showed that $A_\infty \cap E$ is open in the relative topology of $E$. But this contradicts to the connectedness of $E$.
\end{proof}

\section{Skyhooks}
%Let $\Sph_2$ be a 2-dimensional unit sphere. 
In this section we will focus on functions $f:\R^2\to [0,\infty)$. 
%It will be shown later in this section that for the arcwise-connected graph it doesn't really matter whether we consider $\Sph_2$ or $\R^2$.

%\begin{definition}
%Let $f:X\to Y$. We will say that $C$ is a skyhook \footnote{name invented by Tomasz Szynalski} with respect to $f$ if $C$ is homeomorphic with the unit circle, the graph of $f|C$ is connected and $f|C$ is discontinuous in exactly one point (the hook point of $C$ with respect to $f$). 
%\end{definition} 

\begin{definition}
Let $f:X\to [0, +\infty)$. $C$ is a skyhook with respect to $f$, iff there is a homeomorphism from the unit circle $h:\Sph_1\to C$ such that the mapping 
\begin{equation}\label{eskyhook}
[0,2\pi)\ni \phi \to f(h(e^{i\phi})) \text{ is continuous}
\end{equation} 
and $\lim\limits_{\phi\to 2\pi}f(h(e^{i\phi}))=+\infty$. $h(e^0)$ is the infinity point of $C$ with respect to $f$. 
\end{definition}

\begin{definition}
Let $f:X\to [0,+\infty)$. $L$ is a semi-skyhook with respect to $f$, iff $L\subset X$, there is a homeomorphism 
$h:[0,1]\to L$ such that the mapping
\begin{equation}
[0,1)\in t \to f(h(t)) \text{ is continuous}
\end{equation} 
and $\lim\limits_{t\to 1}f(h(t))=+\infty$. $h(1)$ is the infinity point of $L$ with respect to $f$. 
\end{definition}

Whenever the context of $f$ is clear we will skip the part ``with respect to $f$''.\\
We will use specific notation for arcs in an arbitrary space $X$. In the given context $[x,y]$ will denote an arc connecting points $x$ and $y$. Since there are usually many different arcs connecting $x$ and $y$, $[x,y]$ will be considered just as a symbol that denotes chosen arc. If we will say about more than one arc connecting points $x$ and $y$, we will use indexes and prims like $[x,y]_1$ or $[x,y]'$. Once a symbol $[x,y]$ is reserved for some arc in given context, $[v,z]$ can be used only for its sub-arcs. The same rule applies to arcs denotes with indexes and prims respectively. Moreover, since arc, by definition, is homeomorphic with the unit interval, we will assume a natural linear order on a given $[x,y]$ such that $x < y$. We will also use $[x,y)=[x,y]\setminus \{y\}$ and analogously $(x,y]$ and $(x,y)$.

\begin{lemma}\label{simple_skyhook_factory}
If $X$ is a Hausdorff space, $f:X\to [0, \infty)$ has an arcwise-connected graph and an arc $[x,x_0]\subset X$ is a semi-skyhook with the infinity point $x_0$, then there exists a skyhook $C$ with the infinity point $x_0$, such that there exists a semi-skyhook $[x',x_0]\subset[x,x_0]$ with the infinity point $x_0$ and $[x',x_0]\subset C$.
\end{lemma}
\begin{proof}
Since the graph of $f$ is arcwise-connected, we have an arc $[x,x_0]'$ such that $f|[x,x_0]'$ is continuous.
Let $x' = sup\{t\in[x,x_0):t\in [x, x_0]'\}$. Such $x'$ exists as $x\in [x,x_0]$. Since $f|[x,x_0]'$ is bounded and $\lim\limits_{[x,x_0)\ni t\to x_0} f(t) = +\infty$, $x'<x_0$. Since $[x,x_0]'$ is closed, $x'\in [x,x_0]'$. Notice that $[x',x_0]\cap[x,x_0]' = \{x', x_0\}$. Therefore $C = [x',x_0]\cup[x',x_0]'$ is homeomorphic with $\Sph_1$. It's trivial to show that $C$ is a skyhook with the infinity point $x_0$.
\end{proof}

Such a semi-skyhook $[x',x_0]$ with the infinity point $x_0$ will be sometimes named a sub-semi-skyhook of $[x,x_0]$, iff $[x',x_0]\subset [x,x_0]$.

\begin{theorem}\label{force_infinity_point}
If $X$ is locally arcwise-connected at $x$, and $Y$ is a locally compact space, $f:X\to Y$ has a closed graph and $D_f=\{x\}$, then there exists some semi-skyhook $[z,x]\subset X$.
\end{theorem}

\begin{proof} Assume by contradiction that the thesis doesn't hold. Thus for each arc $[z,x]$, $f|[z,x]$ is continuous.
We will apply Lemma \ref{continuous_factory}. Let $A=\{x\}$. Take any open neighbourhood $U$ of $x$. By local arcwise-connectedness at $x$, we have such an open neighbourhood $G$ of $x$, such that for each $x\in G$, $[z,x]\subset G$. Note that $f([z,x])$ is connected. Thus by virtue of Lemma \ref{continuous_factory}, $f$ is continuous at $x$, which contradicts the condition that $x\in D_f$.
\end{proof}

I would like to define the interior of $C\subset \R^2$ homeomorphic with the unit circle. To do this, let mi cite the famous theorem of Schönflies.

\begin{theorem}\label{schonflies}
If $C\subset\R^2$ is homeomorphic with the unit circle, then there exists such a homeomorphism $h:\R^2\to\R^2$ that $h(\Sph_1) = C$.
\end{theorem}

\begin{corollary}\label{circle_disconnect}
If $C\subset\R^2$ is homeomorphic with the unit circle, them it disconnects $\R^2$ into two open subsets, bounded and unbounded.
\end{corollary}

In spite of the above we can formulate the following definition.

\begin{definition}
Let $C\subset \R^2$ be homeomorphic with the unit circle. We will denote by $B(C)$ the bounded open subset from Corollary \ref{circle_disconnect}. Moreover, $\hat{B}(C) = B(C) \cup C$. 
\end{definition}

\begin{corollary}
If $C\subset\R^2$ is homeomorphic with the unit circle, then $B(C)$ is homeomorphic with an open disc, $\hat{B}(C)$ is homeomorphic with a closed disc, $Clo(B(C)) = \hat{B}(C)$, $Int(\hat{B}(C)) = B(C)$ and $\partial B(C) = C$.
\end{corollary}

\begin{corollary}
If $C\subset\R^2$ is homeomorphic with the unit circle, then $C$ is locally connectedly accessible from $B(C)$ and $C$ is locally connectedly accessible from $\R^2\setminus\hat{B}(C)$.
\end{corollary}

(For proof see Example \ref{trivial_example}.)

No we are ready to prove one of the most important facts concerning skyhooks.

\begin{theorem}\label{force_dicont}
If $f:\R^2\to[0, +\infty)$ has a closed and connected graph, $C$ is a skyhook, then $D_f\cap B(C)\not=\emptyset$.
\end{theorem}
\begin{proof}
Assume to the contrary that $B(C)\subset C_f$. Let $x_0$ be the infinity point of the skyhook $C$. By Theorem \ref{neither_or}, $C\setminus\{x_0\}\subset B(C)^f_0$ or $C\setminus\{x_0\}\subset B(C)^f_\infty$. Since $B(C)^f_\infty$ is closed, if $C\setminus\{x_0\}\subset B(C)^f_\infty$, then $C\subset B(C)^f_\infty$, which contradicts (by Theorem \ref{exploding}) to the connectedness of the graph.

Thus $C\setminus\{x_0\}\subset B(C)^f_0$. Since $x_0$ is the infinity point of $C$, it's easy to notice that $x_0\in B(C)^f_\infty$. Since $x_0$ is locally connectedly accessible from $B(C)$, by Corollary \ref{col_exp}, $\lim_{C\setminus\{x_0\}\ni x\to x_0}f(x)=+\infty$. But by the definition of shyhook, we have such $v_n\in C$, that $v_n\to x_0$ and $f(v_n)\to f(x_0)$. Contradiction. 
\end{proof}

\begin{fact}
If $X$ is a locally connected, completely metrisable space and $E\subset X$ is closed, then $E$ is a cutting of $X$ iff $X\setminus E$ is disconnected.
\end{fact}
(for proof: this is a conclusion from \cite[50.II.8]{Ku2}

\begin{fact}\label{clo_lc}
If $X$ is a topological space and $A,B\subset X$ are both closed and locally connected, then $A\cup B$ is locally connected.
\end{fact}
(for proof: e.g. \cite[49.I.3]{Ku2})

\begin{lemma}\label{important_lac}
If $X$ is a topological space, $A\subset X$ is closed and locally arcwise-connected, $E\subset X$ is closed and $\partial E$ is locally arcwise-connected, then $(A \cap E) \cup \partial E$ is locally arcwise-connected. Morover, if $A$ and $\partial E$ are connected and are not separated by any point and $|A\cap\partial E| \geq 2$, then $(A \cap E) \cup \partial E$ is arcwise-connected and not separated by any point.
\end{lemma}
\begin{proof}
Take any $x\in (A \cap E) \cup \partial E$. If $x\in A \cap Int E$, then, since $A$ is locally arcwise-connected, $(A \cap E) \cup \partial E$ is locally arc-wise connected at $x$.
If $x\in \partial E \setminus A$, then the set $(A \cap E) \cup \partial E$ is locally arcwise-connected at $x$, since $A$ is closed and $\partial E$ is locally arcwise-connected. It's enough to check the case $x\in A\cap \partial E$.\\
Take any open neighbourhood $U$ of the point $x$.
By local arcwise-connectedness of $\partial E$ there exists an open neighbourhood $V_{\partial E} \subset U$ of the point $x$, such that for any $y\in V_{\partial E} \cap \partial E$ there exists an arc connecting points $x$ and $y$, contained in $U\cap \partial E$. 
On the other hand, since $A$ is locally arcwise-connected, there exists an open $V_A$ of the point $x$, such that $V_A\subset V_{\partial E} \subset U$ and for any $z\in V_A\cap A$ there exists an arc connecting points $z$ and $x$, contained in $V_{\partial E}\cap A$. 
We will show (*) that for each $p\in (A\cap Int E)\cap V_A$ there exists an arc connecting $p$ with some point in $\partial E$, contained in $(A\cap E)\cap V_{\partial E}$. Indeed, considering how $V_A$ was chosen, there exists an arc $[p,x]\subset V_A\cap A$. 
Since $p\in Int E$ we can define $t_0 = sup\{t\in[p,x]:[p,t]\subset E\}$. Since $p\in Int E$, $t_0 > p$. Obviously $t_0\in \partial E$ and $[p,t_0]\subset V_{\partial E} \cap (A \cap E)$. (*) is shown. 

Take any $q\in ((A \cap E) \cup \partial E)\cap V_{\partial E}$. If $q\in \partial E$, then considering how $V_{\partial E}$ was chosen, there exits an arc connecting points $q$ and $x$, contained in $\partial E \cap U \subset ((A\cap E)\cup \partial E)\cap U$. If $q\in (A\cap Int E)$, then by (*) there is an arc connecting $q$ with some $e\in\partial E$, contained $(A\cap E)\cap V_{\partial E}$. Next, considering how $V_{\partial E}$ was chosen there exists an arc connecting points $e$ and $x$, contained in $((A\cap E)\cup \partial E)\cap U$. Therefore we can show that there exists and arc connecting points $q$ and $x$, contained in $((A\cap E)\cup \partial E)\cap U$.
The proof that $(A\cap E)\cup \partial E$ is locally arcwise-connected is completed.

We will show the moreover part.
If $A$ and $\partial E$ are connected, then they are also arcwise-connected. Thus, it's easy to notice, that since $A \cap \partial E \not=\emptyset$, $(A \cap E)\cup \partial E$ is also arcwise-connected.
Take any $p\in (A \cap E) \cup \partial E$. Since $A\setminus \{p\}$ and $\partial E\setminus \{p\}$ are connected and locally arcwise-connected, the are arcwise-connected. Since $|A\cap\partial E| \geq 2$, $A\setminus\{p\} \cap (\partial E\setminus\{p\}) \not=\emptyset$. Thus $((A \cap E) \cup \partial E)\setminus\{p\}$ is arcwise-connected.  
\end{proof}

\begin{theorem}\label{circle_factory}
If $E\subset \R^2$ is a locally connected continuum, not separated by any point and $R$ is a connected component of $\R^2\setminus E$, then $\partial R$ is homeomorphic with the unit circle.
\end{theorem}
\begin{proof}
Since $E$ is a locally connected continuum not separated by any point, it contains no cut point. Since $\R^2$ is homeomorphic with a sphere without a point, the thesis holds by \cite[61.II.4]{Ku2}.
\end{proof}

\begin{theorem}\label{zero_components}
If $E\subset \R^2$ is a locally connected continuum and $R_1, R_2, \dots$ is an infinite sequence of pair-wise disjoint connected components of $\R^2\setminus E$, then $\lim\limits_{n\to\infty} diam(R_n) = 0$.
\end{theorem}
\begin{proof}
\cite[61.II.10]{Ku2}.
\end{proof}

\begin{lemma}\label{force_skyhook}
If $K\subset\R^2$ is a locally connected continuum, not separated by any point, $\R^2\setminus K$ has at least one bounded connected component, $f : K \to [0,+\infty)$ has a closed graph, $x_0\in K$, $D_f=\{x_0\}$ and $(x_0,f(x_0))\in Clo(f|K\setminus\{x_0\})$, then there exists a bounded connected component $R$ of $\R^2\setminus K$, such that $\partial R$ is a skyhook with the infinity point $x_0$. 
\end{lemma}
\begin{proof}
Let $\mathcal{S}$ be a family of all bounded connected components of $\R^2\setminus K$ (by the conditions of this theorem, is not empty). Let $E=\bigcup\mathcal{S} \cup K$. Notice that $E$ is a complement of the unbounded connected component of $\R^2\setminus K$.   
By Theorem \ref{circle_factory}, $E$ is homeomorphic with a closed disk. Assume by contradiction that thesis doesn't hold. We will construct an extension $\hat{f}:E\to[0,+\infty]$ of $f$, which is continuous everywhere expect of the point $x_0$. 

Let $R$ denote an arbitrary bounded connected component of $\R^2\setminus K$. By Theorem \ref{circle_factory}, $R$ is homeomorphic with an open disc. Since we assumed that thesis doesn't hold and the graph is closed, we have two possibilities: (1) $f|\partial R$ is continuous or (2) $x_0\in \partial R$ and $f|\partial R\setminus \{x_0\}$ is continuous and convergent to $+\infty$ at the point $x_0$. If (1) holds, then $f(\partial R)$ is a closed interval, then by Tietze extension theorem, there exists a continuous extension $h$ of $f$ to $Clo(R)$, such that $h(Clo(R))=f(\partial R)$. Let $\hat{f} = h$ on $R$. If (2) holds, we define a helping function $g$, such that $g=f$ on $\partial R$ and $g(x_0)=+\infty$. Note that $g$ is continuous on $\partial R$ and $g(\partial R) = [a,+\infty]$, then again by Tietze extension theorem, there exists a continuous extension $\hat{g}$ of $g$ to $Clo(R)$, such that $\hat{g}(Clo(R))=g(\partial R)$. Let $\hat{f} = \hat{g}$ on $R$. We defined the function $\hat{f}$ on each bounded connected component of $\R^2\setminus K$. Additionally $\hat{f} = f$ on $K$. Thus we defined $\hat{f}:E\to[0,+\infty]$.\\
We will show that $\hat{f}$ is continuous everywhere except of the point $x_0$. Since $\hat{f}$ is continuous on each bounded connected component of $\R^2\setminus K$, it's enough to show that it's continuous on $K\setminus\{x_0\}$. 
take any $z_0\in K\setminus\{x_0\}$ and any sequence $E\ni z_n \to z_0$. Since $\hat{f}=f$ on $K$, it's enough to assume, that $z_n\in\bigcup\mathcal{S}$. 
If almost all elements of $z_n$ are contained in one bounded connected component of $\R^2\setminus K$, then, by the construction, $\hat{f}(z_n)\to\hat{f}(z_0)$. Now, to prove continuity it's enough to show that $\hat{f}(z_n)\to\hat{f}(z_0)$ for $z_n\in R_{n}$, where $R_{n}$ is a sequence of pair-wise disjoint bounded connected components of $\R^2\setminus K$. 
By Theorem \ref{zero_components}, $\lim\limits_{n\to\infty} diam(R_{n}) = 0$. 
We will construct a sequence $z'_{n}$. Take any index $n$. By the construction of $\hat{f}$, if $\hat{f}(z_n) < +\infty$, we can chose such $z'_n\in \partial R_{n}$, that $\hat{f}(z'_n)=\hat{f}(z_n)$; if $\hat{f}(z_n) = +\infty$, we can chose such $z'_n\in \partial R_{n}$, that $\hat{f}(z'_n) > n$.
Note that $\{z'_n\} \subset K$ and $z'_n\to z_0$, thus $\hat{f}(z'_n)\to\hat{f}(z_0)$. By the construction of $z'_n$, $\hat{f}(z_n)\to\hat{f}(z_0)$. We showed continuity of $\hat{f}$ everywhere expect of the point $x_0$.\\
We will show, that at the point $x_0$ the function $\hat{f}$ has exactly two accumulation points, namely $f(x_0)$ and $+\infty$. Take any sequence $E\ni z_n \to x_0$. If almost all elements of $z_n$ are contained in $K$, since the graph of $f$ is closed, the sequence $\hat{f}(z_n)$ has no more than two accumulation points, $f(x_0)$ and $+\infty$. If almost all elements $z_n$ are contained in some bounded connected component of $\R^2\setminus K$, then by the construction of the extension, $\hat{f}(z_n)\to f(x_0)$ or $\hat{f}(z_n)\to +\infty$.
Thus it is enough to assume that $z_n\in R_{n}$, where $R_{n}$ is a sequence of pair-wise disjoint bounded connected components of $\R^2\setminus K$. We will construct a sequence $z'_n$ in exactly the same way as in the construction above. Obviously $\{z'_n\}\subset K$ and by Theorem \ref{zero_components}, $z'_n\to x_0$. The sequence $\hat{f}(z'_n)$ has no more than two accumulation points: $f(x_0)$ and $+\infty$. But by the construction of $z'_n$, $\hat{f}(z_n)$ has exactly the same accumulation points as $\hat{f}(z'_n)$. We showed so far that at the point $x_0$, the function $\hat{f}$ has no more than two accumulation points: $f(x_0)$ and $+\infty$. Thus we have such an open neighbourhood $U$ of $x_0$, that for each $x\in U\cap E$ we have $\hat{f}(x) < f(x_0) + 1$ or $\hat{f}(x) > f(x_0) + 2$. Since $x_0$ is a point of discontinuity of $f$ and $(x_0,f(x_0))\in Clo(f|K\setminus\{x_0\})$, $\hat{f}$ has exactly two accumulation points at $x_0$: $f(x_0)$ and $+\infty$. Therefore, we can choose $u_1,u_2\in U\cap E$ such that $\hat{f}(u_1) < f(x_0) + 1$ i $\hat{f}(u_2) > f(x_0) + 2$. Since $E$ is homeomorphic with a closed disc and by Theorem \ref{schonflies}, we can as well require that there exists an arc $[u_1, u_2]\subset (U\cap E)\setminus \{x_0\}$.  Note that $\hat{f}$ is continuous on $[u_1,u_2]$. Thus there exists $u\in [u_1, u_2]$, such that $\hat{f}(u) = f(x_0) + 1$. This contradicts that $u\in U\cap E$.                  
\end{proof}

For further applications of the above lemma, it will be convenient to keep in mind the following remark, which follows directly from Theorem \ref{circle_factory}.

\begin{remark}
If $K\subset \R^2$ is locally connected continuum not separated by any point and $Int K = \emptyset$, then $\R^2\setminus K$ has at least one bounded connected component. 
\end{remark}

\begin{lemma}\label{circles_intersect}
If $C_1, C_2\subset \R^2$ are homeomorphic with a circle, $C_1\cap B(C_2) \not=\emptyset$ and $C_2\cap B(C_1) \not=\emptyset$, then $|C_1\cap C_2|\geq 2$. 
\end{lemma}
\begin{proof}
Since $C_2 \cap B(C_1) \not=\emptyset$, $C_1\not\subset B(C_2)$. Thus $C_1 \cap (\RR^2\setminus \hat{B}(C_2)) \not=\emptyset$. Since $\partial B(C_2) = C_2$, by conectedness of $C_1$, $C_1 \cap C_2 \not= \emptyset$. If $C_1 \cap C_2 = \{x\}$, it would mean that $x$ separates $C_1$, contradiction. Thus $|C_1\cap C_2|\geq 2$. 
\end{proof}

\begin{theorem}\label{small_skyhook}
If $f:\R^2\to [0,+\infty)$ has a closed graph, $C_1$ and $C_2$ are skyhooks, then there exists a skyhook $C$, such that $B(C)$ is a connected component of $\R^2 \setminus (C_1 \cup C_2)$ and $C\subset \hat{B}(C_1)$.
\end{theorem}
\begin{proof}
If $B(C_1)$ is a connected component of $\R^2\setminus(C_1\cup C_2)$, then the thesis is trivial. If $B(C_2)$ is a connected component of $\R^2\setminus(C_1\cup C_2)$ and $B(C_1)$ is not, then $B(C_2)\subset \hat{B}(C_1)$. Thus for $C=C_2$ thesis holds.
Therefore, it's enough to check the case: $C_1\cap B(C_2) \not=\emptyset$ and $C_2\cap B(C_1) \not=\emptyset$.
By Lemma \ref{circles_intersect}, $|C_1\cap C_2|\geq 2$. Let $E = (\hat{B}(C_1)\cap C_2)\cup C_1$. By Lemma \ref{important_lac}, $E$ is locally connected continuum not separated by any point. Note that the graph of $f|E$ is closed and locally connected. Let $x_i$ be the point of infinity of $C_i$ for $i=1,2$. If $f|{E\setminus\{x_1\}}$ is continuous, then by Theorem \ref{force_skyhook}, the thesis holds. Assume that $f|{E\setminus\{x_1\}}$ is discontinuous. Then $x_2\in E\setminus\{x_1\}$.  
Obviously there exists $\delta_0 > 0$, such that $x_1\not\in B(x_2,\delta_0)$ and $f|((B(x_2,\delta_0)\cap E)\setminus \{x_2\})$ is continuous. Thus by Lemma \ref{force_infinity_point}, there exists a semi-skyhook $[z,x_2]$ with the infinity point $x_2$, such that $[z,x_2] \subset E\cap B(x_2,\delta_0)$. Since the graph of $f|E$ is arc-wise connected, by Theorem \ref{simple_skyhook_factory}, there exists a skyhook $C_0\subset E$ with the infinity point $x_2$. Let $E' = (\hat{B}(C_0)\cap E)\cup C_0$. By Lemma \ref{important_lac}, $E'$ is locally connected continuum not separated by any point. 
We will show that $f|{E'\setminus \{x_2\}}$ is continuous. Assume by contradiction, that it's discontinuous. So $x_1\in E'\setminus\{x_2\}$. Obviously there exists $\delta > 0$, such that $x_2\not\in \hat{B}(x_1,\delta)$, thus $f|(B(x_1,\delta)\cap E')\setminus \{x_1\}$ is continuous. Therefore, by Lemma \ref{force_infinity_point}, there exists a semi-skyhook $[z',x_1]'\subset E'\cap B(x_1,\delta)$ with the infinity point $x_1$. Since $f|C_2\cap \hat{B}(x_1,\delta)$ is continuous and thus bounded, there exists a semi-skyhook $[y,x_1]'\subset [z',x_1]'\cap C_1$ with the infinity point $x_1$. 
Note that $[y,x_1]'\subset E'\subset \hat{B}(C_0) \subset \hat{B}(C_1)$, thus $[y,x_1]'\subset C_0$. Thus contradiction, since $x_2$ is the infinity point of $C_0$ and $x_1\in E'\setminus\{x_2\}$. We showed that $f|E'\setminus \{x_2\}$ is continuous. Thus, by Theorem \ref{force_skyhook}, there exists a skyhook $C$, such that $B(C)$ is a connected component of $\R^2\setminus E'$. Since $B(C)\subset B(C_0)\subset B(C_1)$, we have $B(C)\cap (C_1\cup C_2) = B(C) \cap (B(C_1)\cap (C_1\cup C_2)) \subset B(C) \cap E = B(C) \cap (B(C_0) \cap E) \subset B(C) \cap E' = \emptyset$. Thus since $C\subset E' \subset C_1\cup C_2$, it's easy to show that $B(C)$ is a connected component of $\R^2\setminus (C_1\cup C_2)$.        
\end{proof}

\begin{corollary}\label{n_small_skyhook}
If $f:\R^2\to [0,+\infty)$ has a closed graph, $C_1, C_2, \dots, C_n$ are skyhooks, then there exists a skyhook $C$ such that $B(C)$ is a connected component of $\R^2 \setminus \bigcup\limits_{i=1}^n C_i$ and $C\subset \hat{B}(C_1)$. 
\end{corollary}
\begin{proof}
Assume that the thesis holds for $n-1$. There exists such a skyhook $C'$ such that $C'\subset \bigcup\limits_{i=1}^{n-1} C_i$ and $B(C')\cap \bigcup\limits_{i=1}^{n-1} C_i = \emptyset$ i $C'\subset \hat{B}(C_1)$. By Theorem \ref{small_skyhook}, there exists a skyhook $C$, such that $C\subset C' \cup C_n$, $B(C) \cap (C' \cup C_n) = \emptyset$ and $C\subset \hat{B}(C')$. Since $B(C)\subset B(C')$, we have $B(C) \cap \bigcup\limits_{i=1}^{n-1} C_i = \emptyset$. On the other hand $B(C) \cap C_n = \emptyset$, thus $B(C) \cap \bigcup\limits_{i=1}^{n} C_i = \emptyset$. Furthermore, $C \subset  \bigcup\limits_{i=1}^{n} C_i$ and $C\subset\hat{B}(C')\subset\hat{B}(C_1)$, thus the thesis holds.     
\end{proof}

\begin{lemma}\label{comp_existence}
If $f:\R^2\to [0, +\infty)$ has a closed, connected and locally connected graph, $D_f$ is a locally connected continuum, not separated by any point, and $x\in D_f$ is an isolated point of discountinuity of the function $f|D_f$, then there exists a connected component $S$ of $\R^2\setminus D_f$, such that $\partial S$ is homeomorphic with the unit circle, $x\in \partial S$, and there exist arcs $[z_1, x], [x,z_2]\subset D_f\cap \partial S$ such that $[z_1, x]$ is a semi-skyhook with the infinity point $x$ and $f|[x,z_2]$ is continuous.  
\end{lemma}
\begin{proof}
First we will show that $(x,f(x))\in Clo(f|(D_f\setminus\{x\}))$. Assume by contradiction that $(x,f(x))\not\in Clo(f|(D_f\setminus\{x\}))$. Chose on open neighbourhood $U$ of $x$, such that $f|(D_f\setminus\{x\})$ is continuous on $Clo(U)\cap(D_f\setminus\{x\})$. By the connectedness of the graph of $f$, we have a sequence $U\ni x_n\to x$, such that $x_n\not=x$ and $f(x_n)\to f(x)$. Obviously almost all $x_n\in U\setminus D_f$. But by Mazurkiewicz-Moore Theorem, the graph of $f$ is locally arcwise-connected. Therefore we can chose some index $k$, such that there exists an arc $[x_k, x]$, such that $f|[x_k,x]$ is continuous. Since $(x,f(x))\not\in Clo(f|(D_f\setminus\{x\}))$, $[x_k, x)\in U\setminus D_f$. Since $[x_k, x)$ is connected, there exists a connected component $R$ of $\R^2\setminus D_f$, such that $x\in \partial R$. By Theorem \ref{circle_factory}, $\partial R$ is homeomorphic with the unit circle. Therefore $x$ is locally connectedly accessible from $R$. Thus by Theorem \ref{exploding}, $x\in R^f_0$. Since $R^f_0$ is open in a relative topology of $Clo(R)$ and by Corollary \ref{easy_con}, $f|Clo(R)$ is continuous on $R^f_0$, we showed $(x,f(x))\in Clo(f|(D_f\setminus\{x\}))$.\\
By Tietze extension theorem we can extend $f|(Clo(U)\cap(D_f\setminus\{x\}))$ to the continuous function $g:D_f\setminus\{x\}\to[0, +\infty)$. Let $\hat{g}:D_f\to [0, +\infty)$ be an extension of $g$ such that $\hat{g}(x) = f(x)$. Thus $\hat{g} = f$ on $Clo(U)\cap D_f$. For the function $\hat{g}$ conditions of Lemma \ref{force_skyhook} hold. Thus there exists a connected component $S$ of $\R^2\setminus D_f$, such that $\partial S$ is a skyhook with respect to $\hat{g}$, which completes the proof.
\end{proof}

\section{Network}

\begin{lemma}\label{too_many}
If $f:\R^2\to[0,+\infty)$ has a closed, connected and locally connected graph, $D_f$ is a non-empty locally connected continuum not separated by any point and $f|D_f$ is discontinuous, then the function $f|D_f$ has infinitely many points of discontinuity.
\end{lemma}
\begin{proof}
Assume to the contrary that there are only finitely many discontinuity points of the function $f|D_f$. Let $\{x_1, \dots, x_n\}$ be the set of these points.\\
Let $R(x,r)$ be a family of all connected components $S$ of $\R^2\setminus D_f$ such that $x\in\partial S$, there exists a semi-skyhook $[z,x]\subset \partial S$ with the infinity point $x$ and $S\not\subset B(x,r)$.
By Theorem \ref{zero_components}, $R(x,r)$ is a finite family for any $x\in \R^2$ and $r > 0$. By Lemma \ref{comp_existence}, we can choose $r_1, \dots, r_n > 0$ such that $R(x_i, r_i)\not=\emptyset$ and $f|D_f$ is continuous on $D_f \cap (\hat{B}(x_i,r_i)\setminus\{x_i\})$ for $i=1,\dots, n$.
Recall that, by Theorem \ref{circle_factory}, boundaries of all connected components of $\R^2\setminus D_f$ are homeomorphic with the unit circle.
Therefore, for each index $i=1,\dots, n$ and $S\in R(x_i, r_i)$, we can choose a semi-skyhooks $\tau^{i,S}_1, \tau^{i,S}_2$ with the infinity point $x_i$, such that $\tau^{i,S}_1, \tau^{i,S}_2\subset \partial S$ and each semi-skyhook $\tau\subset \partial S$ with the infinity point $x_i$ has a common sub-semi-skyhook with  $\tau^{i,S}_1$ or with $\tau^{i,S}_2$ (it might be $\tau^{i,S}_1 = \tau^{i,S}_2$). Moreover, by Lemma \ref{simple_skyhook_factory}, we will require that there exists a skyhook $C^{i,S}_j$ with the infinity point $x_i$, such that $\tau^{i,S}_j\subset C^{i,S}_j$ for $j=1,2$.\\
Let 
$$\mathcal{J} = \bigcup\{C^{i,S}_j:S\in R(x_i, r_i); i \in \{1, \dots, n\}; j\in\{1,2\}\}.$$
By Corollary \ref{n_small_skyhook}, there exists a skyhook $C$ such that $C\subset \mathcal{J}$ and $B(C)$ is a connected component of $\R^2\setminus \mathcal{J}$. By Theorem \ref{force_dicont}, $B(C)\cap D_f\not=\emptyset$. It's easy to notice, that there exists such an index $m$, that $x_m$ is the infinity point of $C$ and there exists $l\in\{1,2\}$ and $Z\in R(x_m, r_m)$ such that there exists a sub-semi-skyhook $s$ of $\tau^{m,Z}_l$ such that $s\subset C$. Thus $|D_f\cap C| \geq 2$. Let $E=(\hat{B}(C)\cap D_f)\cup C$. By Lemma \ref{important_lac}, $E$ is a locally arcwise-connected continuum not separated by any point. Since $C$ is a skyhook with respect to $f$, it's easy to notice that $f|E$ has a closed, connected and locally connected graph. By the argument similar to that from the proof of Lemma \ref{comp_existence}, there exists a function $\hat{g}:E\to [0,+\infty)$ such that $\hat{g} = f$ on $\hat{K}(x_m, r_m)\cap E$ and $\hat{g}$ is continuous on $E\setminus\{x_m\}$. Thus by Lemma \ref{force_skyhook}, there exists a bounded connected component $R$ of $\R^2\setminus E$, such that $x_m\in \partial R$ and $\partial R$ is a skyhook with respect to $\hat{g}$ with the infinity point $x_m$. Since $R$ is bounded, $R\subset B(C)$.\\
We will show that $\partial R$ is a skyhook with respect to $f$. Suppose that it isn't so. Since $\partial R$ is a skyhook with respect to $\hat{g}$, there must exist a discontinuity point $x'$ of $f|\partial R$ such that $x'\not=x_m$. First we will show that (**) if $f|\partial R$ is discontinuous at $x\in \partial R$, then $x\in C$ and $f|D_f$ is discontinuous at $x$. If $x = x_m$ then (**) is trivial. If $x\in B(C)\cap D_f$, then obviously $x$ is a discontinuity point of $f|D_f$, but this leads to the contradiction, for all discontinuity points of $f|D_f$ are contained in $\mathcal{J}$ and $B(C)\cap \mathcal{J} = \emptyset$. Thus $x\in C$. Since $f|C$ is continuous at $x$ and $\partial R\subset C \cup (D_f\cap B(C))$, $f|D_f$ must be discontinuous at $x$. (**) is shown. Now by (**), there exists such an index $k\not=m$, that $x'=x_k$. By (**) $x_k$ is an isolated point of discontinuity of $f|\partial R$. Thus by Theorem \ref{force_infinity_point}, there exists a semi-skyhook $[z, x_k]$ with respect to $f$, such that $[z,x_k]\subset \partial R$.  Since $f|C$ is continuous at $x_k$, there exists sub-semi-skyhook $[u,x_k]$ of $[z,x_k]$ such that $[u,x_k)\subset D_f\cap B(C)$. Since $R$ is a bounded connected component of $\R^2\setminus E$, $R\subset \R^2\setminus D_f$. Take a connected component $R_0$ of $\R^2\setminus D_f$, such that $R\subset R_0$. Since $[u,x_k]\subset \partial R \cap D_f$, $[u,x_k]\subset \partial R_0$. Since $x_m\in \partial R\cap D_f$, $x_m\in \partial R_0$. Therefore, $R_0\not\subset K(x_k, r_k)$. Thus $R_0\in R(x_k, r_k)$. Hence, there exists $a\in \{1,2\}$ such that $\tau^{k, R_0}_a$ has a common sub-semi-skyhook with $[u,x_k]$. So, $C^{k, R_0}_a \cap [u,x_k)\not=\emptyset$ and thus, $\mathcal{J} \cap [u, x_k)\not=\emptyset$. But this means that $B(C)\cap \mathcal{J}\not=\emptyset$, which  contradicts that $B(C)$ is a connected component of $\R^2\setminus\mathcal{J}$. We showed that $\partial R$ is a skyhook with respect to $f$.\\
But as $R\subset \R^2\setminus D_f$, this would violate Therorem \ref{force_dicont}.
\end{proof}

\begin{theorem}\label{exl_div}
If $f:\R^2\to[0,+\infty)$ has a closed and connected graph, $C\subset\R^2$ is homeomorphic with the unit circle, an arc $[x,y]\subset C$, there exists an open set $U$ such that $U\cap D_f = (x,y)$, $A = U\cap B(C)$ and $B=U\cap(\R^2\setminus \hat{B}(C))$, then either $(x,y)\subset A^f_\infty$ or $(x,y)\subset B^f_\infty$.  
\end{theorem}
\begin{proof}
Since the graph of $f$ is closed and $(x,y)\subset D_f$,
$A^f_0 \cap B^f_0 \cap (x,y) = \emptyset$.
Assume to the contrary that $(x,y)\not\subset A^f_\infty$ and $(x,y)\not\subset B^f_\infty$. Hence $(x,y)\cap A^f_0 \not=\emptyset$ and $(x,y)\cap B^f_0 \not=\emptyset$. Thus, it's clear that $(x,y)\cap A^f_0 = (x,y)\cap B^f_\infty$ and $(x,y)\cap B^f_0 = (x,y)\cap A^f_\infty$.
Let $a\in (x,y) \cap A^f_0$ and $b\in (x,y)\cap B^f_0$. By symmetry, assume that $a < b$. Let $\alpha = inf\{t > a : t\in A^f_\infty\}$. Since $A^f_0$ is open in the relative topology of $Clo(A)$ and $B^f_0$ is open in the relative topology of $Clo(B)$, we have $a < \alpha < b$. Note that $(a,\alpha)\subset A^f_0$ and $\alpha \in A^f_\infty$. Since $(a, \alpha)\subset B^f_\infty$, $\alpha \in A^\infty \cap B^\infty$. Let $F = [a,b] \cap A^f_\infty \cap B^f_\infty$. Obviously $F\not=\emptyset$. We will show, that for any $v\in F$ and any sequence $\R^2\setminus F\ni v_n \to v$, we have $f(v_n)\to+\infty$. Since $v\in A^\infty \cap B^\infty$, it's enough to show, that for $(a,b)\setminus F = (a,b)\cap (A^f_0 \cup B^f_0)\ni v_n\to v$, we have $f(v_n)\to+\infty$. But this is a conclusion from Corollary \ref{col_exp}. Thus $f|F$ is closed-open in the relative topology of $f$. This contradicts to the connectedness of the graph.
\end{proof}

\begin{theorem}\label{only_two}
If $f:\R^2\to[0,+\infty)$ has a closed and connected graph, $U$ is an open subset of $\R^2$ and there is an arc $(x,y) = U\cap D_f$, then $f|D_f$ has no more than $2$ point of discontinuity on the arc $(x,y)$.
\end{theorem}
\begin{proof}
Take any $C$ which is homeomorphic with the unit circle and $[x,y]\subset C$. Let $A = U\cap B(C)$ and $B=U\cap(\R^2\setminus \hat{B}(C))$. By Theoren \ref{exl_div}, $(x,y)\subset A^f_\infty$ or $(x,y)\subset B^f_\infty$. By symmetry, assume that $(x,y)\subset B^f_\infty$.\\
We will show that there exist points $x'$ and $y'$ such that $(x',y')=A^f_0\cap (x,y)$ (here for convenience we assume that $(x',x')=\emptyset$).\\ 
It's enough to show that for any $a,b\in (x,y)\cap A^f_0$ ($a<b$), we have $[a,b]\subset A^f_0$. Let $F = [a,b] \cap A^f_\infty$. We will show that $F=\emptyset$. Take any $v\in F$ and $v_n\in (a,b)\setminus F = (a,b)\cap A^f_0$ such that $v_n\to v$. By Corollary \ref{col_exp}, $f(v_n)\to \infty$. Therefore, since $F\subset A^f_\infty\cup B^f_\infty$, $f|F$ is closed and open set in the relative topology of $f$. But since graph of $f$ is connected, $F=\emptyset$. Thus $[a,b]\in A^f_0$.\\
Then we have $(x',y')=A^f_0\cap (x,y)$. Assume that $x\leq x'<y'\leq y$. By Corollary \ref{easy_con}, $f|(x',y')$ is continuous. By Corollary \ref{col_exp}, $\lim_{(x',y')\ni t\to x'}f(t)=+\infty$ and 
$\lim_{(x',y')\ni t\to y'}f(t)=+\infty$. Since $[x,x']\subset A^f_\infty\cup B^f_\infty$ and $[y,y']\subset A^f_\infty\cup B^f_\infty$, the graphs of $f|[x,x']$ and $f|[y,y']$ must be connected. So, in virtue of \cite{Bur}, they must be continuous. Thus $f|(x,y)$ has no more than $2$ points of discontinuity. In the above we assumed that $x\leq x'<y' \leq y$. If $x'=y'$, by our convention, it means that $(x,y)\subset A^f_\infty$. Then $(x,y)\subset A^f_\infty\cup B^f_\infty$ and in virtue of \cite{Bur}, two or more points of discontinuity of $f|(x,y)$ would make the graph of $f$ disconnected.
\end{proof}

After \cite[61.IV]{Ku2} we will define a network on the plane:

\begin{definition}\label{network}
Let $E\subset \R^2$. $E$ is a network, iff
\begin{enumerate}
\item $E = \bigcup\limits_{i=1}^n [a_i, b_i]_i$,
\item $[a_1, b_1]_1\cup [a_2, b_2]_2$ is homeomorphic with the unit circle,
\item $[a_k,b_k]_k \cap \bigcup\limits_{i=1}^{k-1} [a_i, b_i]_i = \{a_k, b_k\}$ for $k=2, \dots, n$.
\end{enumerate}
\end{definition}

\begin{corollary}
Each network is a locally connected continuum not separated by any point.
\end{corollary}

\begin{corollary}\label{circle_again}
If $E\subset\R^2$ is a network, then $\R^2\setminus E$ has finitely many connected components and for each its connected component $R$, $\partial R$ is homeomorphic with the unit circle.
\end{corollary}

\begin{corollary}\label{wrapper_open}
If $E\subset \R^2$ is a network, then there exist some natural number $m$, arcs $[x_i, y_i]_i'$ and open sets $U_i$ for $i=1,\dots,m$ such that $E = \bigcup\limits_{i=1}^m [x_i, y_i]_i'$ and $U_i\cap E = (x_i, y_i)_i'$ for $i=1, \dots, m$.
\end{corollary}

\begin{lemma}\label{almost_final}
If $f:\R^2\to[0,+\infty)$ has a closed and connected graph and $D_f$ is a network, then the graph of $f$ is locally connected and $f|D_f$ has at most finitely many points of discontinuity.
\end{lemma}
\begin{proof}
By Corollary \ref{wrapper_open} and Theorem \ref{only_two}, $f|D_f$ has at most finitely many points of discontinuity.
Take any connected component $R$ of $\R^2\setminus D_f$. By Corollary \ref{circle_again}, $\partial R$ is homeomorphic with the unit circle. The local connectedness of the graph of $f|R^f_0$ is obvious (Corollary \ref{easy_con}). Note that by Corollary \ref{col_exp}, the graph of $f|R^f_0$ is closed. But since $f|D_f$ can have at most finitely many points of discontinuity, $f|\partial R$ has at most finitely many points of discontinuity. Since the graph of $f|\partial R$ is close, this implies that the graph of $f|\partial R$ is locally connected. So by Fact \ref{clo_lc}, the graph of $f|Clo(R)$ is locally connected. By Corollary \ref{circle_again}, closures of all (finitely many) connected components of $\R^2\setminus D_f$ sum up to the whole space $\R^2$. Then again by Fact \ref{clo_lc}, the graph of $f$ is locally connected.
\end{proof}

\begin{theorem}\label{discon}
If $f:\R^2\to[0,+\infty)$ has a closed graph and $D_f$ is a network, then the graph of $f$ is disconnected.
\end{theorem}
\begin{proof}
Since $D_f$ is not empty, $f$ is discontinuous. Assume to the contrary that the graph of $f$ is connected. Then by Lemma \ref{almost_final}, the graph of $f$ is also locally connected and $f|D_f$ has at most finitely many points of discontinuity. So by Lemma \ref{too_many}, $f|D_f$ is continuous. Then for each connected component $R$ of $\R^2\setminus D_f$, by Theorem \ref{neither_or}, $\partial R \subset R^f_0$ or $\partial R \subset R^f_\infty$. But as we assumed, the graph of $f$ is connected, so only $\partial R \subset R^f_0$ is possible. But then, by Corollary \ref{easy_con}, $f|Clo(R)$ is continuous. This, by Corollary \ref{circle_again}, implies that $f$ is continuous.
\end{proof}

\section{Codomain generalisation}

We will prove that the result from the above section for $f:R^2\to [0,+\infty)$ can be easily extended to $f:R^2\to Y$, where $Y$ is a metrisable $\sigma$-locally compact space. To do this let me cite \cite[T1]{Wil}.

\begin{theorem}\label{wiliam}
Let $X$ be a metrisable topological space. Then $X$ is $\sigma$-locally compact, if and only if there exists a metric for which each bounded closed subset is compact. 
\end{theorem}

We will prove the following generalisation.

\begin{theorem}\label{di_gen}
If $X$ is a Hausdorff topological space, $Y$ is a metrlisable $\sigma$-locally compact space, $f:X\to Y$ has a closed graph, then there exists a closed graph function $\hat{f}:X\to [0,+\infty)$ such that $D(f) = D(\hat{f})$ and the graph of $\hat{f}$ is homeomorphic with the graph of $f$.   
\end{theorem}

\begin{proof}
By Theorem \ref{wiliam}, for $Y$ we can choose such a metric $d$ for which each bounded and closed subset of $Y$ is compact. Choose some $y_0\in Y$. Let $\hat{f}(x) = d(y_0,f(x))$.
First we will show that the graph of $\hat{f}$ is closed. Take any $x\in X$ and net $x_\alpha$ such that $(x_\alpha, \hat{f}(x_\alpha))\to (x,\hat{y})$. Since the graph of $f$ is closed and $f(x_\alpha)$ without perhaps some initial set of indices is bounded, $f(x_\alpha)\to f(x)$. Thus $\hat{f}(x_\alpha)\to \hat{f}(x)$. We showed that $\hat{f}$ has a closed graph.\\
It's trivial that if $f$ is continuous at $x$, then $\hat{f}$ is also continuous at $x$.\\
Assume that $\hat{f}$ is continuous at $x$. Take any net $x_\alpha\to x$. We have $\hat{f}(x_\alpha)\to\hat{f}(x)$. Then again since the graph of $f$ is closed and $f(x_\alpha)$ without perhaps some initial set of indices is bounded, we have $f(x_\alpha)\to f(x)$. So, we showed that $D(f) = D(\hat{f})$.\\
Let $H:f\to \hat{f}$ where $H(x,f(x)) = (x, \hat{f}(x))$. By the reasoning similar to the above it can be easily shown that $H$ is continuous and its inverse is also continuous, so $H$ is a homeomorphism.
\end{proof}

Now by applying Theorem \ref{di_gen} to Theorem \ref{discon} we can formulate the following corollary. 

\begin{corollary}\label{big_hit}
If $Y$ is a metrisable $\sigma$-locally compact space, $f:\R^2\to Y$ has a closed graph and $D_f$ is a network, then the graph of $f$ is disconnected.
\end{corollary}


\begin{thebibliography}{99}
\bibitem[Burgess 90]{Bur}{C. E. Burgess, \itshape{Continuous Functions and Connected Graphs}, The American Mathematical Monthly, Vol. 97, No. 4, 337--339, 1990.}
\bibitem[Doboš 85]{Dob}{J. Doboš, \itshape{On the set of points of discontinuity for functions with closed graphs}, Časopis pro pěstování matematiky, Vol. 110, No. 1, 60--68, 1985.}
\bibitem[Engelking 89]{Eng}{R. Engelking, {\itshape General Topology}, Berlin: Heldermann, 1989.}
\bibitem[Jel\'{\i}nek 2003] {JL} {\itshape{J. Jel\'{\i}nek, A discontinuous function
    with a connected closed graph}, Acta Universitatis Carolinae,
    44, No. 2, 73--77, 2003.}
\bibitem[Kuratowski I 66]{Ku1}{K. Kuratowski, {\itshape Topology Volume I}, New York 1966.}
\bibitem[Kuratowski II 66]{Ku2}{K. Kuratowski, {\itshape Topology Volume II}, New York 1966.}
\bibitem[Williamson 87]{Wil}{R. Williamson, L. Janos, \itshape{Constructing Metrics with the Heine-Borel Property}, Proceedings of the American Mathematical Society,
Vol. 100, No. 3, pp. 567-573 (Jul., 1987)}
\bibitem[Wójcik 2004]{wojcik1}{M. R. Wójcik, M. S. Wójcik, \itshape{Separately
    continuous functions with closed graphs}, Real Analysis
    Exchange, Vol. 30, No. 1, 23--28, 2004/2005.}
\bibitem[Wójcik 2007]{wojcik2}{M. R. Wójcik, M. S. Wójcik, \itshape{Characterization of continuity for real-valued
functions in terms of connectedness}, Houston Journal of Mathematics,
Vol. 33, No. 4, 1027--1031, 2007.}
\end{thebibliography}
\end{document}